\documentclass[a4paper]{amsart}

\usepackage[margin=2.5cm]{geometry}
\usepackage[english]{babel}
\usepackage[utf8]{inputenc}
\usepackage[T1]{fontenc}
\usepackage{amsmath,amssymb,amsthm}
\usepackage{graphicx}
\usepackage{tikz}
\usepackage{url}

\def\epsilon{\varepsilon}

\def\NN{\mathbb{N}}
\def\ZZ{\mathbb{Z}}
\def\QQ{\mathbb{Q}}
\def\RR{\mathbb{R}}

\def\diam{\operatorname{diam}}

\def\ua{\underline{a}}

\def\mid{\operatorname{mid}}

\def\HDim{\operatorname{HD}}

\def\Markov{\operatorname{M}}
\def\MarkovSpec{\mathbf{M}}
\def\Lagrange{\operatorname{L}}
\def\LagrangeSpec{\mathbf{L}}

\newtheorem{theorem}{Theorem}
\newtheorem{proposition}[theorem]{Proposition}
\newtheorem{lemma}[theorem]{Lemma}
\newtheorem{corollary}[theorem]{Corollary}
\newtheorem{remark}[theorem]{Remark}

\title{Approximations of the Lagrange and Markov spectra}
\author{Vincent Delecroix, Carlos Matheus and Carlos Gustavo Moreira}

\address{Vincent Delecroix:
Max-Planck Institut
Bonn, Germany
and
Universit\'e de Bordeaux CNRS (UMR 5800)
Bordeaux, France.
}

\email{vincent.delecroix@u-bordeaux.fr}

\address{Carlos Matheus:
 CMLS, \'Ecole Polytechnique, CNRS (UMR 7640), 91128, Palaiseau,
France.
}

\email{matheus.cmss@gmail.com}

\address{Carlos Gustavo Moreira:
School of Mathematical Sciences, Nankai University, Tianjin 300071, P. R. China, and IMPA, Estrada Dona Castorina 110, CEP 22460-320, Rio de Janeiro, Brazil
}

\email{gugu@impa.br}

\date{\today}

\begin{document}

\begin{abstract}
We describe a polynomial time algorithm providing finite sets arbitrarily close (in Hausdorff topology) to the Lagrange and Markov spectra.
\end{abstract}

\maketitle

\section{Introduction}
Given a positive real number $\alpha$ we define the \emph{best constant
of Diophantine approximation of $\alpha$} as
\[
\Lagrange(\alpha) := \limsup_{p,q \to  \infty} \frac{1}{| q (q \alpha - p)|}
\]
where $p$ and $q$ are bound to be positive integers.
The \emph{Lagrange spectrum} $\LagrangeSpec$ is the set $\{\Lagrange(\alpha): \alpha \in \RR \setminus \QQ\}$.
Perron~\cite{Per21} showed that if $\alpha = [a_0; a_1, \ldots]$
is the continued fraction expansion of $\alpha$ then
\[
\Lagrange(\alpha) = \limsup_{n \to \infty} (a_n + [0; a_{n-1}, a_{n-2}, \ldots, a_0] + [0; a_{n+1}, a_{n+2}, \ldots]).
\]
In other words, one can equivalently define the Lagrange spectrum
on the bi-infinite shift $\Sigma = \{1,2,3,\ldots\}^\ZZ$. More concretely, let us define the
\emph{height function} $\lambda_0: \Sigma \to \RR_+$ by
\[
\lambda_0(\ua) = a_0 + [0; a_{-1}, a_{-2}, \ldots] + [0; a_1, a_2, \ldots]
\]
where $\ua = (a_n)_{n \in \ZZ}$. We have
\[
\LagrangeSpec = \left\{\limsup_{n \to +\infty} \lambda_0 \left(\sigma^n(\ua)\right): \ua \in \Sigma \right\}.
\]
where $\sigma: \Sigma \to \Sigma$ is the shift map.
The \emph{Markov spectrum} can be defined similarly as
\[
\MarkovSpec = \left\{\Markov(\ua): \ua \in \Sigma \right\}
\]
where $\Markov(\ua) = \sup_{n \in \ZZ} \lambda_0(\sigma^n(\ua))$.

$\LagrangeSpec$ and $\MarkovSpec$ are subsets of $\RR^+$ which were first systematically studied by Markov~\cite{Ma1},~\cite{Ma2} circa 1879. It is known that the Lagrange, resp. Markov spectrum is the closure of the values of $M(\ua)$ for periodic, resp. ultimately periodic sequences $\ua$
(see \cite[Chapter 3]{CF}): in particular, $\LagrangeSpec$ and $\MarkovSpec$ are closed sets with $\LagrangeSpec \subset \MarkovSpec$. Nevertheless, they do not coincide~\cite{Fr68}. Actually, the second and third authors of the present text proved recently \cite{MatheusMoreira-bound} that the Hausdorff dimension $\HDim(\MarkovSpec \setminus \LagrangeSpec)$ of $\MarkovSpec \setminus \LagrangeSpec$ satisfies
\[
0.53128 < \HDim(\MarkovSpec \setminus \LagrangeSpec) < 0.986927.
\]
Also, it was shown by Freiman~\cite{Fr73b} and Schecker~\cite{Sch77} that $\LagrangeSpec$ contains the half-line $[\sqrt{21},+\infty)$. We recommend consulting Cusick--Flahive book~\cite{CF} for a detailed account of several features of these fascinating spectra describing also the cusp excursions of geodesics on the modular surface.

Our aim in this article is to approximate $\LagrangeSpec$ and $\MarkovSpec$ by mean of computations.
More precisely we will be doing so by constructing finite sets that are close in
Hausdorff distance. Let $X$ and $Y$ be set of real numbers. We say that $X$ and $Y$
are $\epsilon$-close (in Hausdorff distance) if
\[
\forall x \in X, \exists y \in Y, |x - y| < \epsilon
\quad \text{and} \quad
\forall y \in Y, \exists x \in X, |x - y| < \epsilon.
\]
\begin{theorem}
\label{thm:algo}
Let $R > 0$. Then there exists an
algorithm that given $Q$ provide finite sets $(1/Q)$-close respectively to
the Lagrange and Markov spectrum in $[0,R]$. There exists a constant
$0 < d_R < 1$ such that its running time is $O(Q^{3 d_R})$ with the
following upper bounds
\begin{itemize}
\item $d_R < 0.532$ for $R \leq \sqrt{13} \sim 3.606$,
\item $d_R < 0.706$ for $R \leq 2 \sqrt{5} \sim 4.472$,
\item $d_R < 0.789$ for $R \leq \sqrt{21} \sim 4.583$.
\end{itemize}
\end{theorem}

The numbers $d_R$ above are actually
Hausdorff dimensions of the sets $E_K$, $2\leq K\leq 4$, of real numbers whose continued fraction
expansions only contain the digits from $1$ to $K$.

\begin{remark}
As we mentioned already, $\LagrangeSpec$ contains the half line $[\sqrt{21},+\infty)$
and it does not make any sense to consider values of $R$ larger than $\sqrt{21}$.
\end{remark}

Among our several motivations to get Theorem \ref{thm:algo}, we would like to mention that an efficient algorithm producing high resolution drawings of $\LagrangeSpec$ and $\MarkovSpec$ could potentially lead the way to solve some long-standing questions such as Berstein\footnote{In fact, despite the fact that some heuristic arguments from the projection theory of fractal sets are compatible with the presence of intervals in $\LagrangeSpec$ before the so-called Hall's ray, Berstein conjecture is somewhat surprising to us because it propose a concrete \emph{relatively large} interval before Hall's ray.} conjecture~\cite{Be73} that $[4.1,4.52]\subset \LagrangeSpec$. As it turns out, the algorithm provided by Theorem \ref{thm:algo} is not sufficiently powerful yet to put us in good position to attack any of these questions. We hope to pursue this discussion in future works.

Let us now describe the main ideas of the algorithm in Theorem \ref{thm:algo}. Because we consider Lagrange
and Markov values in the interval $[0,R]$ we can restrict to continued fractions
with partial quotients $\{1,2, \ldots, \lfloor R \rfloor\}$. Let
$\Sigma_K := \{1,\ldots,K\}^\ZZ$. We denote by $\LagrangeSpec_K$ and
$\MarkovSpec_K$ respectively the Lagrange and
Markov spectra restricted to the shift $\Sigma_K$, that is
$\LagrangeSpec_K = \left\{\Lagrange(\ua): \ua \in \Sigma_K \right\}$
and $\MarkovSpec_K = \left\{\Markov(\ua): \ua \in \Sigma_K \right\}$.
For each $K$, the relation $\LagrangeSpec_K \subset \MarkovSpec_K$ holds.

The first step of the algorithm consists in constructing a subshift
of finite type $\Sigma_{K,Q}$ on the alphabet $\{1,\ldots,K\}$ depending
on the quality of approximation $Q$. $\Sigma_{K,Q}$ is the set of infinite paths
on a graph $G_{K,Q}$ and each edge of $G_{K,Q}$ corresponds
to a certain cylinder set $\vec{b} = (b_{-m}, \ldots, b_0, \ldots, b_n)$ of the
shift $\Sigma_{K,Q}$ satisfying
\begin{equation}
\label{eq:control}
\sup_{\ua \in \vec{b}} \lambda_0(\ua) - \inf_{\ua \in \vec{b}} \lambda_0(\ua) < \frac{1}{Q}.
\end{equation}
In the above equation and all along the text we identify a finite pointed
word $\vec{b}$ (ie with a distinguished origin $b_0$) and there associated cylinder
set $\{\ua \in \Sigma_K:\, \forall k \in \{-m, \ldots, n\}, a_{k} = b_k\}$.

We turn the graph $G_{K,Q}$ into a weighted graph by considering
on the edge associated to $\vec{b}$ the weight
$\displaystyle \frac{\inf_{\ua \in \vec{b}} \lambda_0(\ua) + \sup_{\ua \in \vec{b}} \lambda_0(\ua))}{2}$.
The weighted graph $G_{K,Q}$ provide an approximation of the shift $\Sigma_K$ together
with the height function $\lambda_0$. The condition~\eqref{eq:control} gives upper bound
on the quality of approximation. The rest of the algorithm consists in studying directly
on the graph $G_{K,Q}$ the discrete analogue of Lagrange and Markov spectrum. The
overall complexity of the algorithm is governed by the size of $G_{K,Q}$. It is not
so suprising that this is related to the Hausdorff dimension of the Lagrange
and Markov spectra.

The algorithm developed in this article has been implemented by the first
author using the computer algebra software SageMath~\cite{SageMath}. Part of the code
has also been optimized using the Python-to-C compiler Cython~\cite{Cython} and the
figures were generated with Matplotlib (see \cite{Matplotlib}). The code
is publicly available at
\url{https://plmlab.math.cnrs.fr/delecroix/lagrange}.

Let us describe the organization of the article. In Section~\ref{sec:background},
we compute some auxiliary intervals containing the sets
$\LagrangeSpec_K$ for $1\leq K\leq 4$. In Section~\ref{sec:graphs}, we describe
the weighted directed graphs $G_{K,Q}$ and their basic properties. After that,
Section~\ref{sec:lagrange-markov-in-graph} provides the discrete analogue
of Lagrange and Markov spectra on a weighted directed graph.
In Section~\ref{sec:justification}, we justify that the discrete spectrum
provide $(1/Q)$-approximation of the original Lagrange and Markov spectra.
The main ingredient for the running time complexity is contained in
Theorem~\ref{thm:control:growth:CKQ} from Section~\ref{sec:complexity}. Finally, for the sake
of comparison with Theorem~\ref{thm:algo}, we consider in Section~\ref{sec:bad-algo}
the alternative idea of approaching the Lagrange
spectrum via periodic orbits in $\{1,\dots, K\}^{\mathbb{Z}}$. Unfortunately
our complexity bounds for the resulting algorithm are very poor: roughly
speaking, we are currently obliged to perform calculations with $4^{Q^4}$
periodic orbits in order to rigorously ensure that we got a $(1/Q)$-dense
subset of $\LagrangeSpec$!

The final Section~\ref{sec:final-pictures} contains some rigurous pictures
obtained via the algorithm from Theorem~\ref{thm:algo}.

\section{Preliminary bounds on $\LagrangeSpec_K$}
\label{sec:background}

In this entire article, $K$ is an integer $1\leq K\leq 4$ except when it is explicitly stated otherwise. Recall that $\LagrangeSpec_K$ is the set of Lagrange values $\Lagrange(\ua)$ when
$\ua$ is restricted to the shift $\Sigma_K = \{1,\ldots,K\}^\ZZ$. It is easy to see that
the smallest and largest values of $\LagrangeSpec_K \setminus \LagrangeSpec_{K-1}$ are
respectively $\Lagrange(\overline{K})$ and $\Lagrange(\overline{1,K})$ where
for a finite word $u$ we denote $\overline{u}$ the associated periodic biinfinite word.

\begin{figure}[!ht]
\begin{center}\includegraphics{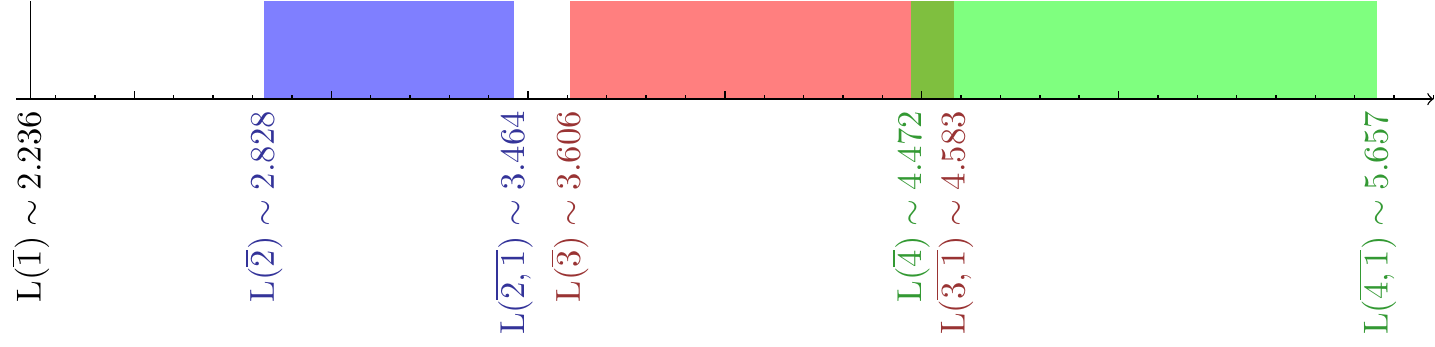}\end{center}
\caption{The lower and upper bounds for $\LagrangeSpec_K \setminus \LagrangeSpec_{K-1}$ when $1\le K\le 4$.}
\label{fig:Lagrange:bounds}
\end{figure}

The points
\begin{align*}
[0; \overline{1}] &= \frac{\sqrt{5}-1}{2}
 & [0; \overline{2}] &= \sqrt{2} - 1
 & [0; \overline{3}] &= \frac{\sqrt{13} - 3}{2}
 &  [0; \overline{4}] &= \sqrt{5} - 2 \\
  & \simeq 0.6180
 && \simeq 0.4142
 && \simeq 0.3028
 && \simeq 0.2361 \\
\end{align*}
and
\begin{align*}
[0; \overline{2,1}] &= \frac{\sqrt{3}-1}{2}
 & [0; \overline{3,1}] &= \frac{\sqrt{21}-3}{6}
 & [0; \overline{4,1}] &= \frac{\sqrt{2}-1}{2} \\
  & \simeq 0.3660
 && \simeq 0.2638
 && \simeq 0.2071 \\
[0; \overline{1,2}] &= \sqrt{3} - 1
 &[0; \overline{1,3}] &= \frac{\sqrt{21}-3}{2}
 &[0; \overline{1,4}] &= 2\sqrt{2}-2 \\
 & \simeq 0.7321
 && \simeq 0.7913
 && \simeq 0.8284
\end{align*}
allow to determine the intervals containing $\LagrangeSpec_K \setminus \LagrangeSpec_{K-1}$ for $1\le K\le 4$, see Figure~\ref{fig:Lagrange:bounds}.

The ranges of $R$ for which we provide upper bounds on $d_R$ in Theorem~\ref{thm:algo}
are visible on Figure~\ref{fig:Lagrange:bounds}
\begin{itemize}
\item $\Lagrange(\overline{2,1}) = \sqrt{13}$ is the maximum of $\LagrangeSpec_2$
\item $\Lagrange({\overline{4}}) = 2 \sqrt{5}$ is the maximum of $\LagrangeSpec_4 \setminus \LagrangeSpec_3$.
\end{itemize}
Let us also mention that $\Lagrange(\overline{3,1}) = \sqrt{21}$ is the maximum of $\LagrangeSpec_3$.

\section{Shifts of finite type}
\label{sec:graphs}
In this section we construct the graphs $G_{K,Q}$ and their associated shifts of finite
type $\Sigma_{K,Q}$. The construction requires intermediate graphs
$G^+_{K,Q}$ that are associated to the one-sided shift
$\Sigma^+_K := \{1,\ldots,K\}^\NN$.

\subsection{$(1/Q)$-cylinders for the Gauss map}
\label{sec:one:sided}
In this section we consider approximations of $\Sigma^+_K$ by shift of
finite type. The continued fraction embedds the shift $\Sigma^+_K$ into $\RR$ as
\[
\begin{array}{ccc}
\Sigma^+_K & \to & \RR_+ \\
\ua = (a_n)_{n \geq 1} & \mapsto & \left[0; a_1, a_2, \ldots\right]
\end{array}
\]
The image of $\Sigma^+_K$ under this map is a Cantor set $E_K$.

For a finite (one sided) cylinder $\vec{b} = (b_1, \ldots, b_n)$ for the Gauss map
we use the notation
\[
\frac{p_k(\vec{b})}{q_k(\vec{b})} = \cfrac{1}{b_1 + \cfrac{1}{b_2 + \ldots + \cfrac{1}{b_k}}}.
\]
A cylinder $\vec{b} \in \{1,\ldots,K\}^n$ projects on $[0,1]$ on
a subset $I_K(\vec{b}) \cap E_K$ where $I_K(\vec{b})$ is the interval with extremities
\[
\frac{p_n(\vec{b}) + \alpha^+_K\ p_{n-1}(\vec{b})}{q_n(\vec{b}) + \alpha^+_K\ q_{n-1}(\vec{b})}
\qquad \text{and} \qquad
\frac{p_n(\vec{b}) + \alpha^-_K\ p_{n-1}(\vec{b})}{q_n(\vec{b}) + \alpha^-_K\ q_{n-1}(\vec{b})}.
\]
where $\alpha^-_K = [0; \overline{K, 1}]$ and $\alpha^+_K = [0; \overline{1,K}]$.
The values of $\alpha^-_K$ and $\alpha^+_K$ for $K \in \{1,2,3,4\}$ were
computed  in Section~\ref{sec:background}. Note that depending on the parity of the length
$|\vec{b}|$ of $\vec{b}$ one or the other is the left handside of the interval.
The diameter of $I_K(\vec{b})$ is
\begin{equation}
\label{eq:diam_cylinder}
\diam_K(\vec{b}) := \frac{(\alpha^+_K - \alpha^-_K)}{(q_n(\vec{b}) + \alpha^+_K\ q_{n-1}(\vec{b})) (q_n(\vec{b}) + \alpha^-_K\ q_{n-1}(\vec{b}))}.
\end{equation}
Note that it is slightly smaller than the \emph{size} $\diam(\vec{b}) = \frac{1}{q_n (q_n
+ q_{n-1})}$ when we do not restrict to the subshift $\{1,\ldots,K\}^\NN$.
Now given $Q$ we consider the following set of cylinders
\begin{equation}
\label{eq:CKQ}
C_{K, Q} := \left\{\vec{b} \in \{1,\ldots,K\}^+:\ \diam_K(\vec{b}) \leq \frac{1}{Q} < \diam_K(\vec{b}') \right\}
\end{equation}
where $\vec{b}'$ denotes the prefix of length $|\vec{b}|-1$ of $b$ and
$\diam_K$ is defined in~\eqref{eq:diam_cylinder}.

First, notice that if $\vec{b} \in C_{K,Q}$ then no proper prefix of
$\vec{b}$ belongs to $C_{K,Q}$ and that any proper suffix of $\vec{b}$
is a prefix of some element in $C_{K,Q}$. The set $C_{K,Q}$ can naturally
be thought as the leaves of a tree rooted at the empty word
$\epsilon$ and where the edges correspond to adding a letter to the right.
Namely, consider the set $\overline{C}_{K,Q}$ to
be the union of $C_{K,Q}$ and all prefixes of elements
of $C_{K,Q}$. The tree has vertex set $\overline{C}_{K,Q}$ and
we put an oriented edge from $u$ to $v$ if $u$ is the prefix of length
$|v|-1$ of $v$.

We add edges on this tree corresponding to the so called
\emph{suffix links}. For each $w \in C_{K,Q}$ we add an
edge from $w$ to its suffix of length $|w|-1$. As we already mentioned,
this suffix belongs necessarily to $\overline{C}_{K,Q}$.
One can visualize the tree and the suffix
links of $C_{2,20}$ on Figure~\ref{fig:tree_and_graph}.

\begin{figure}[!ht]
\begin{center}%
\includegraphics{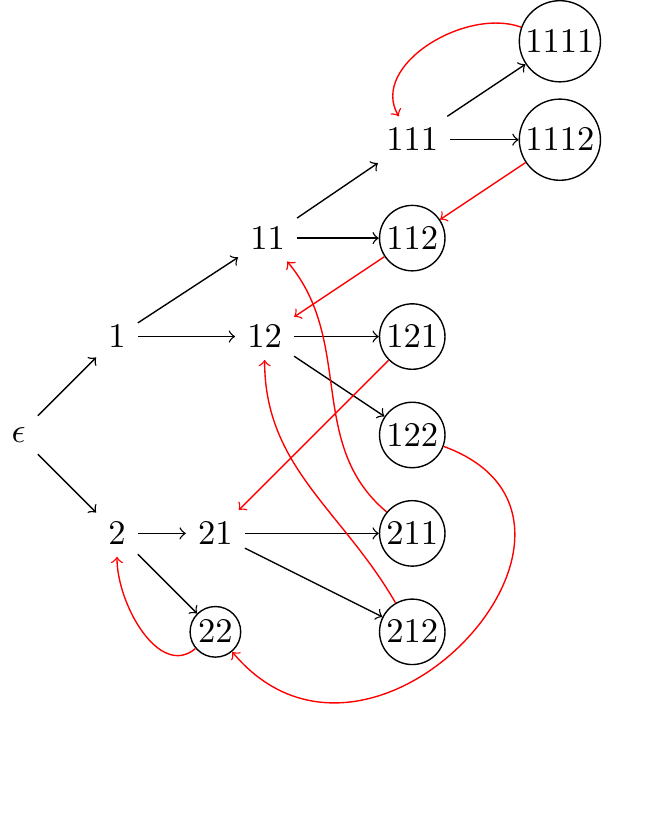}%
\hspace{0.5cm}%
\includegraphics{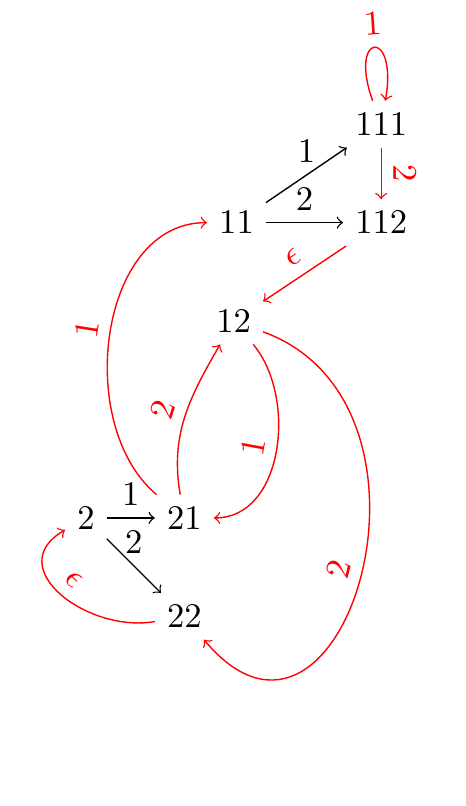}%
\end{center}
\caption{The tree of $T_{2,20}$ and the suffix links (in red) on the left.
The graph $G^+_{2,20}$ on the right with the prolongation edges in black
and the shift edges in red.}
\label{fig:tree_and_graph}
\end{figure}

Now we define the set $V^+_{K,Q}$ as the set of endpoints of the suffix links
(in other words the maximal non-trivial suffixes of elements of $C_{K,Q}$).
We consider two kinds of edges on the vertex set $V^+_{K,Q}$.
First, for each (oriented) path in the tree $T_{K,Q}$ between
pairs of vertices $V^+_{K,Q}$ we add an edge in $G^+_{K,Q}$. We call such
edge a \emph{prolongation edge} (they are in black on
Figure~\ref{fig:tree_and_graph}). Secondly, for each suffix link $s$ from
$u$ to $v$ in $T_{K,Q}$, there is a unique vertex $w$ in $V^+_{K,Q}$ and a path
from $w$ to $u$ that avoids any other element from $V^+_{K,Q}$. We add an edge
from $w$ to $v$ that we call a \emph{shift edge}.

Each edge carries a label that is a finite word on $\{1,\ldots,K\}$ (possibly
empty). They are directly induced from the tree $T_{K,Q}$ for which each edge
carry a letter. Each prolongation edge in $G^+_{K,Q}$ already carries a letter
and we keep this letter as a label. Each shift edge is made of the
concatenation of a path $w \to u$ and a suffix link and we associate to
this edge the label of the path $w \to u$ in the tree $T_{K,Q}$.

\begin{lemma}
For any $K$ and $Q$ the graph $G^+_{K,Q}$ recognize the shift $\{1,\ldots,K\}^\ZZ$:
for any biinfinite word $w \in \Sigma_K$ there exists a unique biinfinite path
$\gamma$ in $G^+_{K,Q}$ so that $w$ can be read along $\gamma$.
\end{lemma}

\begin{remark}
The construction of the graphs $G^+_{K,Q}$ from $C_{K,Q}$ can be generalized
to any set of words with the same properties (every word in $\Sigma_K^+$ has
a prefix in the set and no proper prefix of an element of the set is
contained in the set). If we had used instead of $C_{K,Q}$ the set of words
of given combinatorial length we would have obtain the de Bruijn graph \cite{deBruijn}.
\end{remark}

\subsection{From $G^+_{K,Q}$ to $G_{K,Q}$}
Now that we have the graph $G^+_{K,Q} = (V^+_{K,Q}, E^+_{K,Q})$ at hand we explain
the construction of $G_{K,Q}$.
Similarly to $G^+_{K,Q}$, the graph $G_{K,Q}$ has two kinds of
edges: prolongation edges and shift edges. The shift edges are in
bijection with $C_{K,Q} \times \{1,\ldots,K\} \times C_{K,Q}$.
Let $(p,a_0,s) \in C_{K,Q} \times \{1,\ldots,K\} \times C_{K,Q}$.
Let $u$ and $v$ be the source and target of the shift edge
associated to $s$ in $G^+_{K,Q}$. By construction, if
$s = (s_1, \ldots, s_n)$ then $u = (s_1, \ldots, s_k)$ and
$v = (s_2, \ldots, s_n)$ for some $0 \leq k \leq n$.
The source of the edge corresponding to $(p,a_0,s)$ in $G_{K,Q}$
is $(p,a_0,s)$ and its target is $(p', s_1, v)$ where $p'$
is the unique element in $C_{K,Q}$ that is a prefix of
$a p$.

We now describe prolongation edges. To each prolongation edge
in $G^+_{K,Q}$ from $u = (s_1, \ldots, s_k)$ to $v = (s_1, \ldots, s_{k+1})$
we associate for each $p$ in $C_{K,Q}$ and each $a_0$ in $\{1,\ldots,K\}$ a prolongation
edge from $(p,a_0,u)$ to $(p,a_0,v)$.

For the shift edge corresponding to $(p,a_0,s)$ we associate the weight
\begin{equation}
\label{eq:GKQ:weights}
F((p,a_0,s)) = a_0 + \mid_K(p) + \mid_K(s)
\end{equation}
where $\mid_K(\vec{b})$ is the middle of the interval determined by a cylinder
$\vec{b} = (b_1, \ldots, b_n)$ given by
\begin{equation}
\label{eq:mid_cylinder}
\mid_K(\vec{b}) = \frac{1}{2} \left(
\frac{p_n(\vec{b}) + \alpha^+_K\ p_{n-1}(\vec{b})}{q_n(\vec{b}) + \alpha^+_K\ q_{n-1}(\vec{b})}
+
\frac{p_n(\vec{b}) + \alpha^-_K\ p_{n-1}(\vec{b})}{q_n(\vec{b}) + \alpha^-_K\ q_{n-1}(\vec{b})}.
\right).
\end{equation}
We give weight $0$ to each prolongation edge.

As before, the biinfinite paths on edges of $G_{K,Q}$ define
a subshift of finite type.
\begin{lemma}
For each $Q$, the graph $G_{K,Q}$ recognizes the subshift $\Sigma_K$.
Moreover, for any shift edge in $G_{K,Q}$ associated to the cylinder $(p,a_0,s)
\in C_{K,Q} \times \{1,\ldots,K\} \times C_{K,Q}$ we have
\[
\sup_{\ua \in (p,a_0,s)} | \lambda_0(\ua) - F((p,a,s)) | \leq \frac{1}{Q}.
\]
\end{lemma}

\begin{proof}
Recall that the weights defined in~\eqref{eq:GKQ:weights} are in between the
extremal possible values of $\lambda_0(\ua)$ in the cylinder $(p,a_0,s)$. But
$p$ and $s$ are in $C_{K,Q}$ constructed in Section~\ref{sec:one:sided} and
were chosen so that the corresponding image under the continued fraction map
have diameter $<1/Q$. Hence for each of $\mid_K(p)$ and $\mid_K(s)$
we are off by at most $1/(2Q)$ so that $a_0 + \mid_K(p) + \mid_K(s)$ is off
by at most $1/Q$.
\end{proof}

\begin{remark} Since the Gauss map $g(x)=\{1/x\}$ has derivative $g'(x)=-1/x^2$, we have $(\alpha_K^-)^{2n} \leq \frac{\diam_K(b)}{(\alpha_K^+ -\alpha_K^-)}\leq (\alpha_K^+)^{2n}$ for any $b\in\{1,\dots, K\}^n$.

In particular, an alternative way of constructing a subshift would have been to pick
all possible cylinder of combinatorial length $n$ (where $n$ is chosen
large enough so that all diameters are smaller than $(1/Q)$), but this
would have lead to a larger set of cylinders.
\end{remark}

\section{Lagrange and Markov edges in weighted directed graphs}
\label{sec:lagrange-markov-in-graph}

Let $G$ be a weighted directed graph. We denote $V(G)$ and $E(G)$ respectively
the vertices and edges of $G$ and $w: E(G) \to \RR$ the weight function. The codomain
of the weight function needs not be $\RR$, any totally ordered set would do.

We call an edge $e$ in $G$ to be a \emph{Lagrange edge} if there exists a cycle
$\gamma$ in $G$ that passes through $e$ and so that the weight of the edge $e$ is
maximal among the weights of edges in $\gamma$. An edge is called a
\emph{Markov edge} if there exist two cycles $\gamma^-$ and $\gamma^+$ and a
path $p$ from $\gamma^-$ to $\gamma^+$ so that the edge $e$ is maximal among
the weights of edges in $\gamma^- \cup p \cup \gamma^+$. The definition is
illustrated with a simple example on Figure~\ref{fig:lagrange_markov_edges}.
\begin{figure}[!ht]
\begin{center}\includegraphics{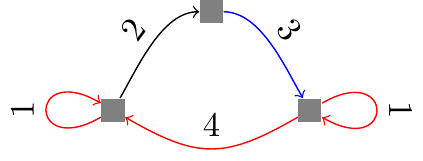}\end{center}
\caption{A weighted graph with its Lagrange edges in red and its single Markov but
not Lagrange edge in blue.}
\label{fig:lagrange_markov_edges}
\end{figure}

A simple approach for computing these edges is to test for each edge
whether it is Lagrange or Markov.
\begin{theorem}
\label{thm:testing_one_edge}
Given $G$ a directed graph and $w: E(G) \to \RR$ and an edge $e$ of $G$.
Determining whether $e$ is Lagrange or Markov has complexity $O(m)$
where $m$ is the number of edges in $G$.
\end{theorem}

\begin{proof}
Let $e$ be an edge and $u$ and $v$ its source and target. Then $e$ is a
Lagrange edge if and only if there is a path from $v$ to $u$ with
maximum edge weight $w(e)$. To compute that, one can simply do a depth-first
search on edges with weight not greater than $w(e)$. Hence testing
whether a single edge is Lagrange is $O(m)$. Now $e$ is a Markov edge
if one can build a path connected to a cycle (both backward and forward).
Similarly, one can detect such cycle with two depth first searches. In
both cases, the search is bounded by the number of edges in the graph $G$.
\end{proof}

As a consequence of Theorem~\ref{thm:testing_one_edge} the complexity of
computing all Lagrange and Markov edges in a given graph $G$ has complexity
$O(m^2)$ where $m$ is the number of edges in $G$. We now describe a procedure
to reduce the computational time for the search of Lagrange and Markov
edges based on online cycle detection and strongly connected component
maintenance~\cite{HKMST08,HKMST12,BFGT}.

\begin{theorem}
Computing the set of weights of Lagrange edges or the set of weights of
Markov edges in a directed weighted graph $G$ can be achieved in
$O(m^{3/2})$ where $m$ is the number of edges of $G$.
\end{theorem}

\begin{proof}
Order the edges in $G$ by non-decreasing weight. Namely
$E(G) = \{e_1, \ldots, e_m\}$ with $w(e_i) \leq w(e_{i+1})$.
We define a sequence of acyclic graphs $(G^{(k)})_{k=0,\ldots,m}$ by considering
the graph obtained from the edges $\{e_1, \ldots, e_k\}$, identifying
the vertices that belong to a same strongly connected component\footnote{Recall that a directed graph is strongly connected whenever there are oriented paths joining any given pair of vertices and a strongly connected component of a directed graph is a maximal strongly connected directed subgraph.} and
removing the loops (edges from a vertex to itself).
The graph $G^{(k)}$ is concretely obtained from $G^{(k-1)}$ by adding the
$k$-th edge and possibly identifying vertices in a newly appeared
strongly connected component. As shown in~\cite{HKMST08,HKMST12,BFGT},
maintaining the strongly connected components in dynamical graph,
or equivalently computing the sequence $G^{(k)}$, can be done in $O(m^{3/2})$.

Now the weights of Lagrange edges are exactly the weights of edges $e_k$
that create a cycle when they are added in $G^{(k)}$. This shows that it
comes at no additional cost. For Markov edges however, one needs to make
a further traversal of the graph. However this can be reduced to a total
cost of $m$ from which we obtain the same upper bound
$O(m^{3/2} + m) = O(m^{3/2})$. In order to do so, one needs to maintain
two flags for each edge: whether it can reach in forward and backward
directions a non-trivial strongly connected component. Updating these
flags is only performed when a new strongly connected component is
detected and each edge is at most traversed once.

Now, having this extra information an edge $e_k$ is Markov if and
only if at the time it is added in $G^{(k)}$ it is such that
its target can reach a strongly connected component in forward
direction and its source can reach a strongly connected component
in backward direction.
\end{proof}

\section{Approximation of Lagrange and Markov spectra}
\label{sec:justification}
Recall that we defined graphs $G_{K,Q,R}$ in Section~\ref{sec:graphs}
and Lagrange and Markov edges in a weighted directed graph were
defined in Section~\ref{sec:lagrange-markov-in-graph}. The main aim of
this section is to prove the following result.
\begin{theorem}
\label{thm:lagrange:markov:edges:approximate}
For any $K,Q$ the set of weights of respectively Lagrange and Markov
edges in $G_{K,Q}$ is $1/Q$-close to respectively $\LagrangeSpec_K$
and $\MarkovSpec_K$.
\end{theorem}

\begin{proof}
We do the proof for the Markov spectrum, the case of Lagrange being similar.

Let $\ua \in \Sigma$ and let $m := \Markov(\ua) = \sup_{n \in \ZZ} \lambda_0(\sigma^n(\ua))$.
We will show that there is a Markov edge whose weight is $1/Q$-close
to $m$. Let $e_n$ be the shift edge corresponding to $\sigma^n(\ua)$
in the graph $G_{K,Q}$. This sequence of shift edges determine a biinfinite
path $\gamma$. Let $e'$ be the edge in $\gamma$ with maximal weight.
We claim that $e'$ is a Markov edge of the graph $G_{K,Q}$ and that
$|m - F(e')| < 1/Q$.

By construction, we have $|\lambda_0(\sigma^n(\ua)) - F(e_n)| < 1/Q$. Hence
the weight $F(e')$ of the supremum satisfies the required bound. Let $n_0$
be any index such that $e_{n_0} = e'$. Since $\gamma$ is biinfinite, there
is a smallest $k$ such that the path $e_{n_0}, e_{n_0+1}, \ldots, e_{n_0+k}$
intersects itself. That is, there is $0 \leq k' < k$ with $e_{n_0+k'} = e_{n_0+k}$.
Since the weights on the edges between indices $n_0+k'$ and $n_0+k$ are
at most $F(e_{n_0})$ we constructed a cycle $\gamma^+$ all of whose weights
are at most $F(e_{n_0})$ that can be reached from $e_{n_0}$ by a path
with weights not larger than $F(e_{n_0})$. The construction of a path
$\gamma^-$ in backward direction is performed similarly and concludes
the fact that $e'$ is a Markov edge in $G_{K,Q}$.
\end{proof}

\section{The size of $C_{K,Q}$ and algorithm complexity}
\label{sec:complexity}

As we saw in Section~\ref{sec:lagrange-markov-in-graph} the time complexity
of detecting Lagrange and Markov edges in a graph is polynomial in the size
of the graph. In this section we provide an upper bound on the size
of the graphs $G_{K,Q}$. It will prove the polynomial bounds in
Theorem~\ref{thm:algo}.

\begin{lemma}
\label{lem:growth:CKQ:and:GKQ}
The number of edges in $G_{K,Q}$ is bounded by
\[
|C_{K,Q}| \left(|C_{K,Q}| + \log_2 \left( \frac{(K(K+1)+1)(K+2)}{K} Q \right) \right).
\]
\end{lemma}

To prove the lemma we need to intermediate results that will also
be used later. The first one gives a control on the diameter $\diam_K(b)$
in terms of the combinatorial length.
\begin{lemma}[\cite{CF}, Lemma 2 p. 2]
\label{lem:CF:diameter}
Let $\alpha=[a_0; a_1,\dots, a_n, a_{n+1},\dots]$ and $\beta=[a_0; a_1,\dots, a_n, b_{n+1},\dots]$ with $a_{n+1}\neq b_{n+1}$. Then $|\alpha-\beta|<1/2^{n-1}$.
\end{lemma}

The second provides a lower bound on the diameter for elements in $C_{K,Q}$.
\begin{lemma}
\label{lem:CKQ:low:bound:diameter}
For any $Q,K$ and any $\vec{b}\in C_{K,Q}$ we have
\[
\frac{K}{(K(K+1)+1)(K+2)}\cdot\frac{1}{Q}\leq \diam_K(\vec{b}).
\]
\end{lemma}

\begin{proof}
On one hand, we have $\frac{\diam_K(\vec{b})}{\diam_K(\vec{b}')}=\frac{(q_{n-1}+\alpha_K^+q_{n-2})(q_{n-1}+\alpha_K^-q_{n-2})}{(q_n+\alpha_K^+ q_{n-1})(q_n+\alpha_K^- q_{n-1})}$ for each $\vec{b}\in C_{K,Q}$. On the other hand, we have that
$$(\alpha_K^+\alpha_K^-)^{-1} (q_n+\alpha_K^+ q_{n-1})(q_n+\alpha_K^- q_{n-1})<((K+1)q_n+q_{n-1})((1+1/K)q_n+q_{n-1})$$
and
$$(\alpha_K^+\alpha_K^-)^{-1}(q_{n-1}+\alpha_K^+q_{n-2})(q_{n-1}+\alpha_K^-q_{n-2})>(Kq_{n-1}+q_{n-2})(q_{n-1}+q_{n-2}).$$
Since
$$(\tfrac{K+1}{K}q_n+q_{n-1}) = (\tfrac{K+1}{K}a_nq_{n-1}+\tfrac{K+1}{K}q_{n-2}+q_{n-1})\leq (K+2)(q_{n-1}+q_{n-2})$$
and
$$((K+1)q_n+q_{n-1}) = ((K+1)a_nq_{n-1}+(K+1)q_{n-2}+q_{n-1})\leq \tfrac{(K+1)K+1}{K}(Kq_{n-1}+q_{n-2}),$$
we obtain that
$$\frac{\diam_K(\vec{b})}{\diam_K(\vec{b}')} > \frac{K}{(K(K+1)+1)(K+2)}.$$

The estimate follows.
\end{proof}

\begin{proof}[Proof of Lemma~\ref{lem:growth:CKQ:and:GKQ}]
Recall that the graph $G_{K,Q}$ has two kind of edges: prolongation edges
and shift edges. The shift edges are in bijection with
$C_{K,Q} \times C_{K,Q}$ which provides the term in $|C_{K,Q}|^2$.

Now we need to bound the number of prolongation edges and we claim
that there number is at most
\[
|C_{K,Q}| \, \log_2 \left( \frac{(K(K+1)+1)(K+2)}{K} Q \right).
\]
This follows from Lemma~\ref{lem:CF:diameter} and the lower bound estimates
on $\diam_K(\vec{b})$ from Lemma~\ref{lem:CKQ:low:bound:diameter}.
\end{proof}

It now remains to estimate the size of the sets $C_{K,Q}$. As we will see next,
the growth rate in $Q$ is intimately linked to the Hausdorff dimension of
the Lagrange and Markov spectra. To do so we apply the techniques in~\cite{Moreira-geom}
and Palis--Takens book~\cite{PaTa} to make this relation precise.

\begin{figure}[!ht]
\begin{center}\includegraphics{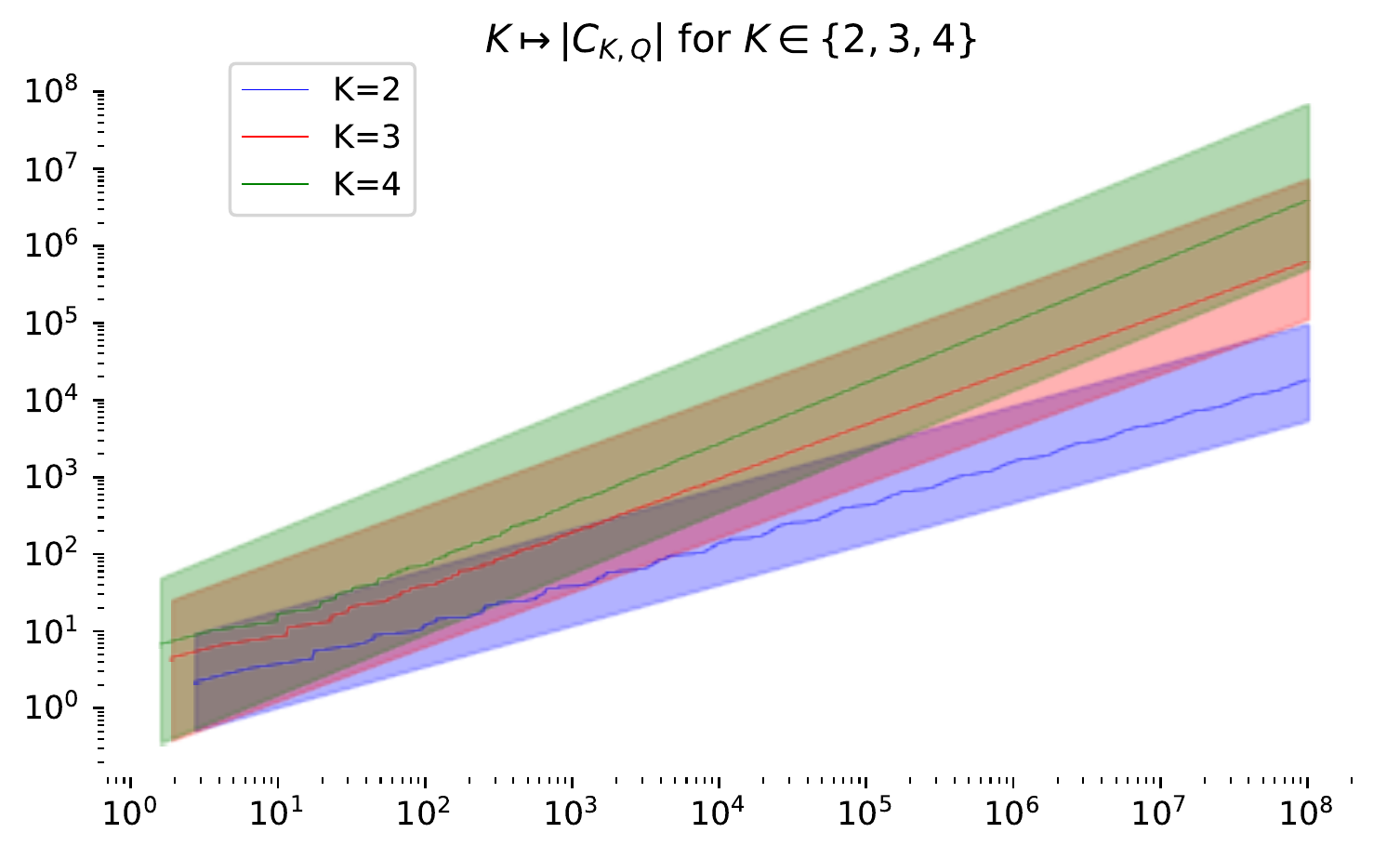}\end{center}
\caption{Sizes of the set $C_{K,Q}$ in log scale. The lighter color represents
the bound we obtain from Theorem~\ref{thm:control:growth:CKQ} together
with the estimates on $\HDim(E_K)$ (from~\cite{Je04,JePo01,JePo18}).}
\end{figure}

Let us recall that we defined the sets $E_K$ in Section~\ref{sec:graphs}
as the image of the one-sided shift $\Sigma^+_K$ under the continued
fraction map. More concretely
\[
E_K := \{[0; a_1, a_2, \ldots]: a_i \in \{1,\ldots,K\}
\subset [0,1].
\]
The set $E_K$ is a closed invariant subset of the Gauss map.

It is known that the Hausdorff dimension $\HDim(E_K)$ of $E_K$ is strictly
increasing in $K$ and $\HDim(E_K) =
1-\frac{6}{\pi^2}K-\frac{72}{\pi^4}\frac{\log K}{K^2}+O(1/K^2)\to 1$ as $K \to
\infty$ (cf. Hensley~\cite{H} and~\cite{He}). For our purposes, it is useful
to know that for small values of $K$ we have the following estimates
(from~\cite{Je04,JePo01,JePo18})
\begin{align*}
0.5312 &< \HDim(E_2) < 0.5313 \\
0.7056 &< \HDim(E_3) < 0.7057 \\
0.7889 &< \HDim(E_4) < 0.7890
\end{align*}

\begin{theorem}
\label{thm:control:growth:CKQ}
There exist constants $c_1(K)$ and $c_2(K)$ such that for any positive
integer $Q$ we have
\[
c_1(K) Q^{\HDim(E_K)} \leq \left| C_{K,Q} \right| \leq c_2(K) Q^{\HDim(E_K)}
\]
where the constants $c_1(K)$ and $c_2(K)$ can be explicitly computed
\[
\begin{array}{lll}
K & c_1(K) & c_2(K) \\
2 & 0.28 & 4.98 \\
3 & 0.23 & 14.85 \\
4 & 0.23 & 31.2 \\
\end{array}.
\]
\end{theorem}

Putting together the estimates from Lemma~\ref{lem:growth:CKQ:and:GKQ} and
Theorem~\ref{thm:control:growth:CKQ} we obtain an estimate on the size.
\begin{corollary}
\label{cor:control:growth:GKQ}
We have $|G_{K,Q}| = O(Q^{2\, \HDim(E_K)})$ for $K=2, 3, 4$.
\end{corollary}

\begin{proof}[Proof of Theorem~\ref{thm:control:growth:CKQ}]
Observe that the $|\vec{b}|$-th iterate of the Gauss map sends $I_K(\vec{b})$ to $[\alpha_K^-, \alpha_K^+]$. Thus, the average of its derivative $(\alpha_K^+-\alpha_K^-)/\diam_K(\vec{b})$ belong to the interval
$$(\alpha_K^+-\alpha_K^-)Q \leq \frac{(\alpha_K^+-\alpha_K^-)}{\diam_K(\vec{b})} \leq (\alpha_K^+-\alpha_K^-) \frac{(K(K+1)+1)(K+2)}{K} Q.$$
Here we used the lower bound estimate from Lemma~\ref{lem:CKQ:low:bound:diameter}.

Since the distortion of the iterates of the Gauss map (i.e., the ratio between its maximal and minimal derivatives) is $\leq 4$ (see, e.g., Proposition 2 in \cite{Moreira-geom}), the maximal derivative $\Lambda_K(\vec{b})$ of the $|\vec{b}|$-th iterate of the Gauss map on  $I_K(\vec{b})$ is $\leq 4(\alpha_K^+-\alpha_K^-) \frac{(K(K+1)+1)(K+2)}{K} Q$ and the minimal derivative $\lambda_K(\vec{b})$ of the $|\vec{b}|$-th iterate of the Gauss map on  $I_K(\vec{b})$ is $\geq \frac{1}{4}(\alpha_K^+-\alpha_K^-)Q$.

As it is explained in pages 68 to 70 of Palis--Takens book \cite{PaTa}, one has
$$\sum\limits_{\vec{b}\in C_{K,Q}}\Lambda_K(\vec{b})^{-\HDim(E_K)}\leq 1\leq \sum\limits_{\vec{b}\in C_{K,Q}}\lambda_K(\vec{b})^{-\HDim(E_K)}.$$ In particular, it follows that
$$c_1(K):=\left(\frac{1}{4}(\alpha_K^+-\alpha_K^-)\right)^{\HDim(E_K)}\leq \frac{|C_{K,Q}|}{Q^{\HDim(E_K)}}\leq \left(4(\alpha_K^+-\alpha_K^-) \frac{(K(K+1)+1)(K+2)}{K}\right)^{\HDim(E_K)}:=c_2(K).$$
This completes the proof of the desired theorem.
\end{proof}

\section{Approximation by periodic orbits}\label{sec:bad-algo}
In this section we consider the following
question. How well the periodic (resp. ultimately periodic) sequences in $\Sigma_K$ with period at most $N$
approximate the set $\LagrangeSpec_K$ (resp. $\MarkovSpec_K$)?

\begin{proposition}\label{p.L}
Let $K \in \{2,3,4\}$ and $N\in\mathbb{N}$. Then the subset
\[
\bigcup\limits_{k=1}^{K^{2N+1}}\{\Lagrange(\overline{u}): u \in \{1,\ldots,K\}^{k}\}
\]
is $1/2^{N-2}$-dense in $\LagrangeSpec_K$. Similarly, the subset
\[
\bigcup\limits_{k=1}^{2\cdot 4^{2N+1}-1}\{\Markov(\overline{p}\,u\,\overline{s}): pus \in \{1, 2, 3,4\}^{k}\}
\]
is $1/2^{N-2}$-dense in $\MarkovSpec_4$.
\end{proposition}

\begin{proof}
The main ingredient in our estimates is Lemma~\ref{lem:CF:diameter}.

\textbf{Lagrange spectrum.}
Fix $K \in\{2, 3, 4\}$. Let $\ua=(a_n)_{n\in\ZZ}\in \Sigma_K$ and put $t=\Lagrange(\ua)\in L$.

By definition, $t ={\limsup}_{j\to\infty}(\alpha_j + \beta_j)$. Hence, given $\delta >0$, there are infinitely many $n \in \NN$ such that $|\alpha_n + \beta_n-t| < \delta$. Also, there is a sequence $h_m\to+\infty$ as $m\to\infty$ such that $\alpha_j + \beta_j<t+1/m$ for all $j\geq h_m$. Fix $N\in\mathbb{N}$. Given an index $j$ consider the finite sequence with $2N+1$ terms $(a_{j-N},\dots, a_j,\dots, a_{j+N}) =: S(j)$. There is a
sequence $S$ such that $S(j)=S$ for infinitely many values of $j$, i.e., there are $j_1 < j_2 <\dots$
with $S(j_i) = S$, $\forall\,i \ge 1$. Note that we may (and do) assume that $\lim_{i\to\infty} (\alpha_{j_i} + \beta_{j_i})=t$.

For each $m\in\mathbb{N}$, consider the subshift of $\Sigma $ given by $\Sigma_{t+1/m}=\{\underline\lambda\in\Sigma \mid m(\underline\lambda)\le t+1/m\}$. Note that $\Sigma_{t+1/m}$ is invariant by transposition operation $(\lambda_n)_{n\in\ZZ}\mapsto (\lambda_{-n})_{n\in\ZZ}$ and by the shift map $\sigma((\lambda_n)_{n\in\ZZ}):=(\lambda_{n+1})_{n\in\ZZ}$. Moreover, it is contained (and well approximated) by the following subshift of finite type: let $B$ be the set of all factors of size $2N+1$ of all elements of $\Sigma_{t+1/m}$, and denote by $\widetilde\Sigma_{t+1/m, N}$ the set of all infinite words in $\Sigma$ whose factors of size $2N+1$ are all in $B$.

We can describe $\widetilde\Sigma_{t+1/m,N}$ as a (not necessarily transitive) Markov shift: the allowed transitions are of the type $(c_0, c_1, c_2,\dots, c_{2N}, c_{2N+1})$ with $(c_0,c_1,\dots,c_{2N})$ and $(c_1,c_2,\dots,c_{2N+1})$ belonging to $B$.

By definition, we have $S\in B$, and it is possible to connect $S$ to itself in $\widetilde\Sigma_{t+1/m,N}$ by a sequence of allowed transitions. The minimum number $k$ of transitions needed to connect $S$ to itself in $\widetilde\Sigma_{t+1/m,N}$ is trivially bounded by the size of $B$, which is at most $K^{2N+1}$.

In other words, for each $m\in\mathbb{N}$, there is $1\leq k\le K^{2N+1}$ and a factor $\omega_m=(\tilde{a}_0,\tilde{a}_1,\dots,\tilde{a}_{2N+k})$ of size $2N+k+1$ of an element of $\widetilde\Sigma_{t+1/m,N}$ with $(\tilde{a}_0,\dots,\tilde{a}_{2N})=(\tilde{a}_k,\dots,\tilde{a}_{2N+k})=S$.

In particular, if $\tilde\theta_m=\overline{(\tilde{a}_0,\dots,\tilde{a}_{k-1})}$ is the periodic sequence of period $\tilde{a}_0\dots \tilde{a}_{k-1}$, then Lemma~\ref{lem:CF:diameter} ensures that $|\Lagrange(\tilde\theta_m)-t|<1/m+2/2^{N-1}$. Since $k\le K^{2N+1}$, there is a periodic sequence $\overline{\mu}$ with period $\leq K^{2N+1}$ such that $\overline{\mu}=\tilde\theta_m$ for infinitely many $m\in\NN$ and, a fortiori, $|\Lagrange(\overline{\mu})-t|\leq 2/2^{N-1}$.

\textbf{Markov spectrum.} Let $\underline\theta=(b_n)_{n\in\ZZ}\in \Sigma$, consider $s=m(\underline{\theta})\in L$, and suppose that $b_n\leq 4$ for all $n\in\ZZ$. We may (and do) also assume that $s=\alpha_0+\beta_0$.

By the pigeonhole principle, there are $1\le i<j\le 4^{2N+1}$ and $1\le u<v\le 4^{2N+1}$ such that $(b_i,b_{i+1},\dots,b_{i+2N})=(b_j,b_{j+1},\dots,b_{j+2N})$ and $(b_{-u},b_{-u-1},\dots,b_{-u-2N})=(b_{-v},b_{-v-1},\dots,b_{-v-2N})$.

By Lemma~\ref{lem:CF:diameter}, if $\hat\theta$ is the doubly pre-periodic sequence
$$\overline{(b_{-v+1},b_{-v+2}\dots,b_{-u-1}b_{-u}})b_{-u+1}\dots b_{i-1}\overline{(b_i,b_{i+1}\dots,b_{j-1})},$$
then $|m(\hat\theta)-s|<2/2^{N-1}$. Notice that the total size of the block formed by the central block and the periods on both sides is at most $2\cdot 4^{2N+1}-1$.
\end{proof}

\begin{remark} The Markov spectrum $\MarkovSpec$ is also characterized by the values of real indefinite binary quadratic  forms (see \cite[Chapter 1]{CF}). An attempt to draw some portions of $\MarkovSpec$ using certain binary quadratic forms of bounded heights was performed by T. Morrison~\cite{Mo12}, but unfortunately this text does not discuss the quality of the approximation of $\MarkovSpec$ obtained by this method.
\end{remark}

\newpage
\section{High resolution Lagrange spectra $\LagrangeSpec_2$ and $\LagrangeSpec_3$}\label{sec:final-pictures}
\begin{figure}[!ht]
\includegraphics[angle=90]{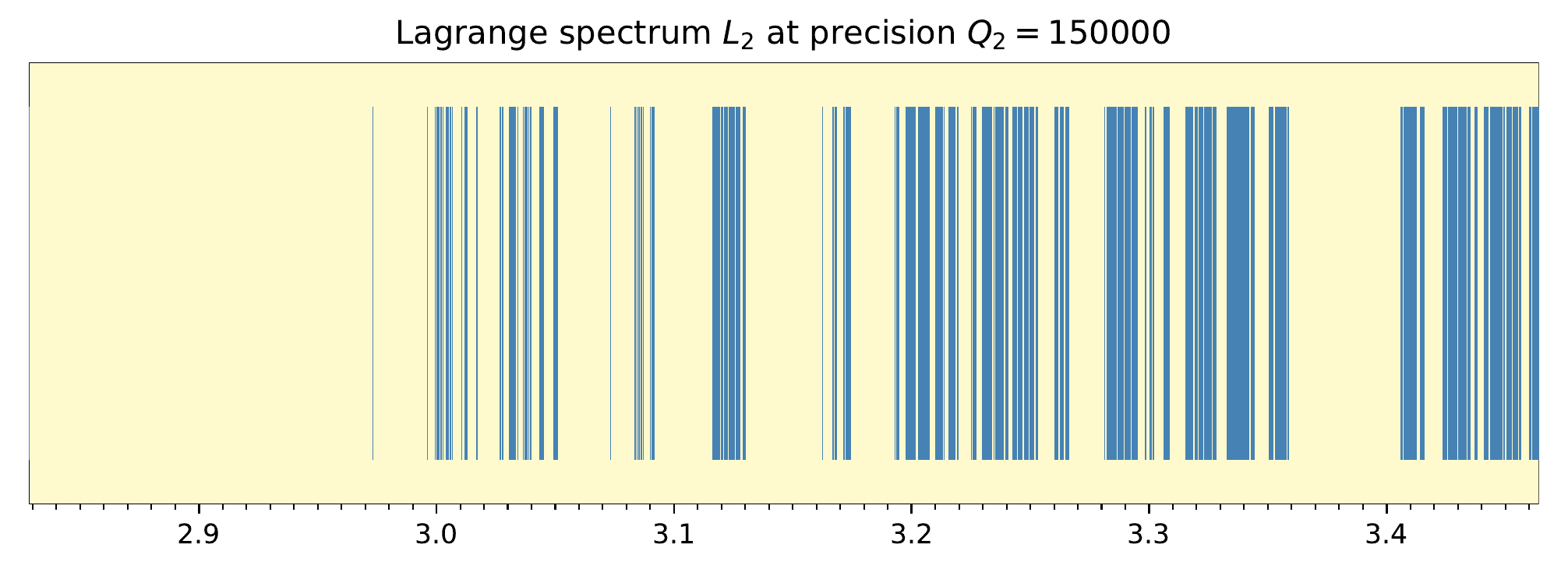}
\hspace{0.4cm}
\includegraphics[angle=90]{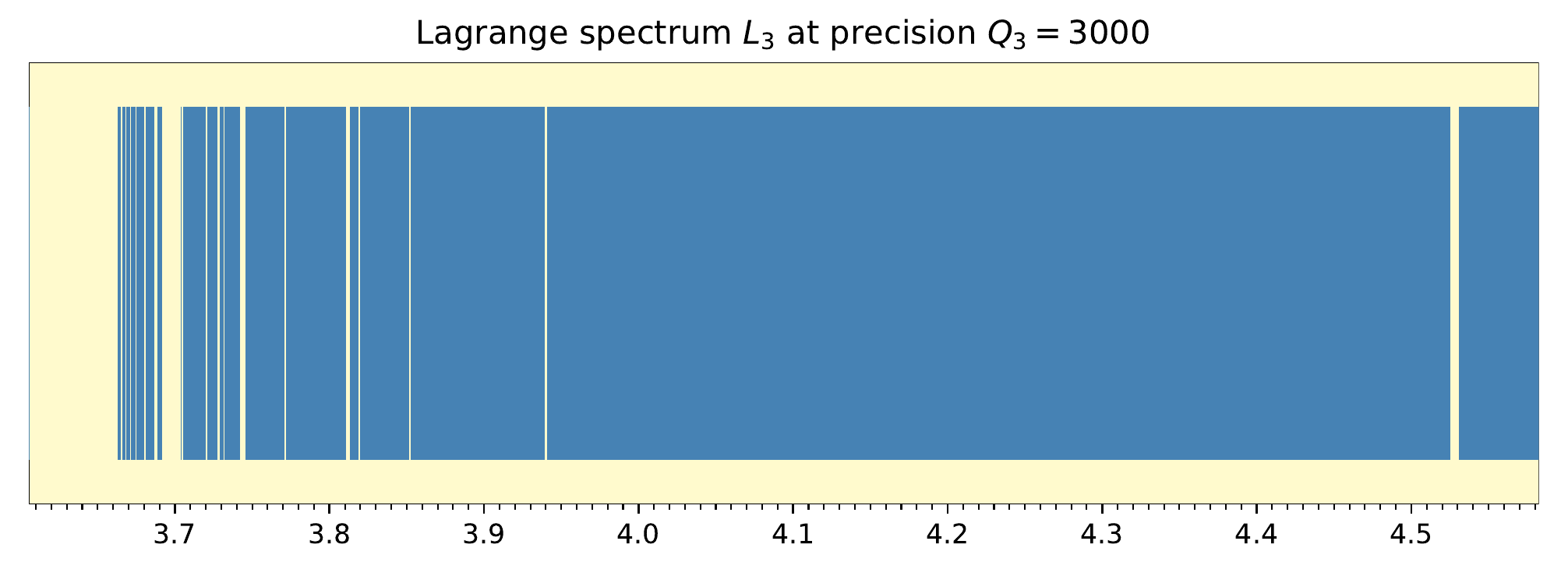}
\caption{Pictures of Lagrange spectra $\LagrangeSpec_2$ and $\LagrangeSpec_3$ obtained from our algorithm.
The parameters $Q_2$ and $Q_3$ are so that the Lagrange spectra $\LagrangeSpec_2$ and $\LagrangeSpec_3$
are respectively at most at Hausdorff distance $1/Q_2$ and $1/Q_3$ from the union
of blue intervals.}
\label{fig:lagrange:spectra}
\end{figure}

\end{document}